\documentclass[12pt]{amsart}
\usepackage{amsmath}
\usepackage{dsfont}
\usepackage{mathrsfs}
\usepackage{graphicx}
\newtheorem{theorem}{Theorem}[section]
\newtheorem{proposition}[theorem]{Proposition}
\newtheorem{definition}[theorem]{Definition}

\newtheorem{remark}[theorem]{Remark}
\newtheorem{lemma}[theorem]{Lemma}
\usepackage{amssymb}
\usepackage{enumerate}
\usepackage{lscape}
\usepackage{tikz}  % Cargar la librería TikZ
\usetikzlibrary{positioning}
\usetikzlibrary{trees}
\usepackage{mathabx}
\usepackage{color}
\definecolor{brass}{rgb}{0.8, 0.25, 0.33}
\allowdisplaybreaks

\newcommand{\N}{\mathbb{N}} %% Conjunto naturales:     \N
\newcommand{\R}{\mathbb{R}} %% Conjunto reales:        \R
\begin{document}
	
	%\begin{frontmatter}
	
	\title[Dynamics of the translation semigroup on directed metric trees]{Dynamics of the translation semigroup on directed metric trees}
	
	%\date{\today}
	
	

	\author[Mangino]{Elisabetta Mangino}
	\address{ Dipartimento di Matematica e Fisica “Ennio De Giorgi”, Università del Salento, Via per Arnesano, I-73100 Lecce, Italy}
	\email{elisabetta.mangino@unisalento.it}
	
		\author[Vargas-Moreno]{\'Alvaro Vargas-Moreno}
	\address{Institut Universitari de Matemàtica Pura i Aplicada,\newline\indent Universitat Polit\`{e}cnica de Val\`{e}ncia, \newline\indent 46022, Val\`{e}ncia, Spain.}
	\email{alvarmo1@upv.es}
	

	\begin{abstract} The dynamics of the left translation semigroup $\{T_t\}_{t \geq 0}$ on weighted $L^p$ spaces over a directed metric tree $L(G)$ is investigated. Necessary and sufficient conditions on the weight family $\rho$ for the strong continuity of the semigroup are provided. Furthermore,  hypercyclicity and weak mixing properties  are characterized in terms of the asymptotic decay of $\rho$ along the tree structure. These results generalize classical $L^p$ translation semigroup dynamics to a graph setting.
	\end{abstract}
	
	\thanks{The first author is member of G.N.A.M.P.A. of the Italian Istituto Nazionale di Alta Matematica (INdAM). 
    This paper has been partially supported by the project MUR PRIN2022 D53D23005580006 “Elliptic and parabolic problems, heat kernel estimates and spectral theory” and by MUR-PRIN 2022 PNRR, Project: Stochastic Modeling of Compound Events, (No. P2022KZJTZ) funded by European Union Next Generation EU.\\
     The second author is supported by MCIN/AEI/10.13039/501100011033/FEDER, UE, Project PID2022-139449NB-I00, which has financed his research stay at Università del Salento. \\
     The authors thank the anonymous reviewers for their insightful comments and technical suggestions, which have significantly improved the readability and clarity of the paper.}
	
	\maketitle
	
	\section*{Introduction}

The study of dynamical systems on metric graphs has its roots in the analysis of the transport equation on so-called quantum graphs. In this context, the effects of the flow along the edges are modeled by differential equations and coupled through specific transmission conditions at the vertices. The pioneering works of Kramar and Sikolya \cite{kramar-sikolya} and Dorn et al. \cite{dorn, dorn-2009, dorn-2010} initiated a systematic analysis of these problems. The monograph by Mugnolo \cite{mugnolo} represents a fundamental reference in this field (see also \cite[Chapter 18]{Batkai-Kramar-Rhandi}). 
A first approach to the study of hypercyclicity in this setting can be found in \cite{proscovia}.

In parallel, the research community turned to the dynamics of operators on graph structures. In particular, the work of Mart\'inez-Avendaño \cite{martinez-avendaño} inaugurated the study of the dynamic properties of the backward shift operator on directed  trees. More recent efforts by Grosse-Erdmann and Papathanasiou \cite{karl-dimitris-hypercyclicity, karl-dimitris-chaos}  have significantly advanced the understanding of hypercyclicity and chaos phenomena in this specific dynamical context. Further related results are due to Menet and Papathanasiou \cite{menet-papathanasiou}, Lopez-Mart\'inez and Papathanasiou \cite{lopez-papathanasiou} and Kawamura \cite{Kawamura}.

Inspired by both the semigroup approach to transport and the rich dynamical analysis of shift operators on trees, this article extends the notion of left translation to the setting of weighted $L^p$ spaces over directed metric trees and studies its dynamics. The first section introduces the structure of metric trees, summarizes the essential definitions of strongly continuous semigroups ($C_0$-semigroups) and hypercyclicity, and establishes the theoretical framework for our work.

The central results of this article are twofold. The first, presented in Section 2, establishes the necessary and sufficient conditions on the weight family $\rho$  that guarantee the translation operator generates a strongly continuous semigroup on the weighted spaces $L^p_\rho$. The second key result, developed in Section 3,  provides  necessary and sufficient conditions on the weight $\rho$ that ensure the semigroup exhibits hypercyclicity on both rooted and unrooted directed metric trees.

\section{Preliminaries}
\subsection{Strongly continuous semigroups}
 We recall some basic definitions and results, referring to the monograph \cite{engel-nagel} for further details.
\begin{definition}\label{def1.1}
			Let $X$ be a Banach space. A one-parameter family $\{T_t\}_{t\geq 0}$ of bounded operators on $X$ is a strongly continuous semigroup  (or a $C_0$-semigroup) if the following conditions are satisfied:
   \begin{enumerate}[(i)]
       \item $T_0=I$
       \item $T_{t+s}=T_t\circ T_s$ for any $s,t\geq 0$
       \item $\lim_{s\rightarrow t}T_s x=T_t x$ for all $x\in X$ and $t\geq 0$.
   \end{enumerate}
   \end{definition}
As a consequence of the strong continuity of the semigroup, there is an exponential bound for the norm of the operators, namely
 there  exist $M\geq 1$ and $w\in\R$  such that 
$$\|T_t\|\leq Me^{wt}, \qquad t\geq 0.$$
We can also express condition $(iii)$ in some equivalent ways. 
\begin{proposition}\label{strongly-condition}\cite[Ch. 1, Proposition 5.3]{engel-nagel} Let $\{T_t\}_{t\geq 0}$ be a family of operators on $X$ satisfying (i) and (ii) of Definition \ref{def1.1}. Then the following assertions are equivalent: 
\begin{enumerate}[(i)]
    \item $\lim_{s\rightarrow t}T_s x=T_t x$ for all $x\in X$ and $t\geq 0$;
    \item $\lim_{t\rightarrow 0}T_t x=x$ for all $x\in X$. 
    \item $\{T_t\}_{t\geq 0}$ is locally equicontinuous and there exists a dense subset $X_0\subset X$ such that $\lim_{t\rightarrow 0}T_t x=x$ for all $x\in X_0$.
\end{enumerate}
\end{proposition}

\subsection{Linear dynamics.} The main references for linear dynamics are the monographs \cite{Alfred, bayart}.
Throughout this section, let $X$ be a separable Banach space. 
An operator $T$ on $X$ is said to be \emph{hypercyclic} if there is some $x\in X$ whose orbit under $T$, 
$$\text{Orb}(x, T):=\{x, Tx, T^2x, ...\},$$
is dense in $X$.  The operator $T$ is said to be \emph{topologically transitive} if for every pair $U, V$ of nonempty open sets there exists some $n\in \N$ such that $T^n(U)\cap V\neq  \emptyset$. Furthermore, if $T^k(U)\cap V\neq  \emptyset$ for all $k\geq n$, then the operator is said to be \emph{mixing}. For separable Banach spaces, the Birkhoff Transitivity Theorem \cite[Theorem 2.19]{Alfred} states that hypercyclicity is equivalent to topological transitivity.
Moreover, if the operator defined by 
$$(T\oplus T)(x, y)=(Tx, Ty),\qquad  (x, y)\in X\times X$$
is hypercyclic, then $T$ is called \emph{weakly mixing}. Every weakly mixing operator is hypercyclic, but the converse fails. We refer to \cite{Alfred} for further insights about the equivalence of the weakly mixing property and the Hypercyclicity Criterion.

All these dynamical properties have a counterpart in the context of strongly continuous semigroups. A strongly continuous semigroup $\{T_t\}_{t\geq 0}$ on   $X$ is said to be \emph{hypercyclic} if there exists some $x\in X$ whose orbit, defined as 
$$\text{Orb}(x, \{T_t\}):=\{T_tx\,:\, 
t\geq 0\},$$
is dense in $X$. Furthermore, the semigroup $\{T_t\}_{t\geq 0}$ is said to be \emph{topologically transitive} if for any pair of nonempty open sets $U, V$ there is some $t_0\geq 0$ such that $T_{t_0}(U)\cap V\neq  \emptyset$, while, if $T_t(U)\cap V\neq  \emptyset$ for any $t\geq t_0$ then the semigroup is said to be \emph{mixing}.  

Topological transitivity is  equivalent to hypercyclicity also in the context of strongly continuous semigroups \cite[Ch. 7]{Alfred}. 

The semigroup $\{T_t\}_{t\geq 0}$ is called \emph{weakly mixing} if $(T_t\oplus T_t)_{t\geq 0}$ is \emph{topologically transitive} in $X\times X$. Every weakly mixing semigroup is hypercyclic, but  to the best of the authors' knowledge, whether the converse holds true remains an open problem. 
A discretization of a semigroup $\{T_t\}_{t\geq 0}$ is a sequence of operators $(T_{t_n})_n$ with $t_n\rightarrow\infty$. A discretization is said to be autonomous if there is some $t_0$ such that $t_n=nt_0$ for every $n\in\N$.
%The following Theorem establishes an equivalence for the weakly mixing property in terms of the autonomous discretizations of the semigroup.  
%\begin{theorem}\cite{Alfred}\label{weaklymixing-autonomous-discretization} Let $\{T_t\}_{t\geq 0}$ be a strongly continuous semigroup on $X$. Then the following assertions are equivalent: 
%    \begin{enumerate}[(i)]
%        \item $\{T_t\}_{t\geq 0}$ is weakly mixing;
%        \item some autonomous discretization of $\{T_t\}_{t\geq 0}$ is weakly mixing;
%        \item every autonomous discretization is weakly mixing. 
%    \end{enumerate}
%\end{theorem}
The following theorem provides an equivalence between the weak mixing property and the existence of a mixing discretization.  
\begin{theorem}\label{weaklymixing-mixingdiscretization}\cite[Propositions 7.20, 7.25]{Alfred}
    Let $\{T_t\}_{t\geq 0}$ be a strongly continuous semigroup on $X$. Then the following assertions are equivalent. 
    \begin{enumerate}[(i)]
        \item $\{T_t\}_{t\geq 0}$ is weakly mixing;
        \item some discretization of $\{T_t\}_{t\geq 0}$ is mixing;
        \item some discretization of $\{T_t\}_{t\geq 0}$ is weakly mixing.
         \item some autonomous discretization of $\{T_t\}_{t\geq 0}$ is weakly mixing;
        \item every autonomous discretization of $\{T_t\}_{t\geq 0}$ is weakly mixing. 
    \end{enumerate}
\end{theorem}  
Finally, we recall the following important result due to Conejero, M\'uller and Peris.

\begin{theorem}\label{hypercyclic-autonomous}\cite[Theorem 6.8]{Alfred}
    Let $\{T_t\}_{t\geq 0}$ be a strongly continuous semigroup on $X$. Then the following assertions are equivalent. 
    \begin{enumerate}[(i)]
        \item $\{T_t\}_{t\geq 0}$ is hypercyclic;
        \item every operator $T_t$ is hypercyclic for $t> 0$
        \item there exists $t_0>0$ such that the operator $T_{t_0}$ is hypercyclic.
    \end{enumerate}
\end{theorem}  

\subsection{Directed metric trees}

In  this section, we  establish the fundamental structure for our investigation, namely the \emph{directed metric tree}. This construction starts with a directed tree, which is then endowed with a continuous metric structure.

We recall that a   \emph{directed tree} is a connected graph $G=(V, E)$ without cycles. That is:
\begin{enumerate}[(i)]
 \item $V$ is a countable set of vertices.
 \item $E\subset (V\times V) \setminus \{(v,v): v\in V\}$ is the set of directed edges (or arcs).
 \item For every $u,v\in V$, $u\neq  v$ there exists a finite sequence of vertices $v_1, \dots, v_k\in V $ such that $(u, v_1)$, $(v_1,v_2)$, $(v_2, v_3)$, \dots, $(v_{k-1}, v_k)$, $(v_k, v) \in E$.
 \item  There are no closed sequences of vertices $v_1, \dots, v_k \in V$ such that $(v_1,v_2)$, $(v_2, v_3)$, \dots, $(v_{k-1}, v_k)$, $(v_k, v_1) \in E$.
 \item For each vertex $v \in V$, there is at most one vertex $w\in V$ such that $(w, v)\in E$. In this case, $w$ is the \emph{parent} of $v$ and $v$ is a \emph{child} of $w$.
\end{enumerate}
A vertex without children is called a \emph{leaf}.

We obtain the \emph{directed metric tree} $L(G)$ by assigning a coordinate $x_e\in [0,1)$ to every edge $e\in E$, increasing in the direction of the edge. Thus, we identify each edge $e$ with the interval $[0,1)$. For an edge $e=(v,w)$, $v$ is the \emph{tail} and $w$ is the \emph{head}.

A directed metric tree is \emph{rooted} if there exists a vertex, called the \emph{root}, that is not the head of any edge (i.e., no edge enters the root). If no such vertex exists, the tree is said to be \emph{unrooted}.

The geometric structure of $L(G)$ can be described  by the edge-to-edge incidence matrix $\mathcal{A}$. Indexing the set of edges as $E=\{e_i\}_{i\in I}$, the entries of $\mathcal{A}$ are defined by:
$$
\mathcal{A}_{ij}=
\begin{cases}
1 & \text{ if the head of } e_i \text{ is the tail of } e_j, \\
 0 & \text{ otherwise.}
\end{cases}
$$

Since the tree structure dictates that each vertex has at most one parent, it holds that 
$$
\forall j \in I:\quad \text{card}(\{i\in I\,: \mathcal A_{ij}\neq 0\})\leq 1.
$$
That is, each column of $\mathcal{A}$ contains at most one non-zero term.

The powers of $\mathcal{A}$ track paths along the tree: $(\mathcal A^n)_{ij}\neq 0$ if and only if there exists a directed path of length $n$ starting with edge $e_i$ and ending with edge $e_j$. We denote the set of indexes of edges reachable from $e_i$ in $n \in\mathds N$ steps as:
$$M_n(i)=\{ j \in I\, :\, (\mathcal A^n)_{ij}\neq 0\},$$
while we set $M_0(i):=\{i\}$.

Observe that, due to the geometry of the tree, for any $n\in\N$, 
\begin{equation}\label{int}M_n(i)\cap M_n(j)=\emptyset \qquad \text{ if } i\neq  j.\end{equation}

We now define the functional space on $L(G)$ on which the dynamics will be studied. A function $f$ on $L(G)$ is identified with the family of its restrictions to the edges, $(f_i)_{i\in I}$, where $f_i = f|_{e_i}: [0,1) \to \mathbb{K}$ (with $\mathbb{K}=\mathbb{R}$ or $\mathbb{C}$).

A weight on $L(G)$ is a family $\rho=(\rho_{i})_{i\in I}$, where $\rho_i:[0,1)\rightarrow \R$,  $\rho_i\in L^1((0,1))$  and is  a.e. strictly positive.  

For a given exponent $p \in [1, \infty)$, we define the weighted space on a single edge $e_i$ as:
$$L^p_{\rho_i}[0,1):=\{ f:[0,1)\rightarrow \R\,:\, f \mbox{ is measurable and } |f|^p\rho_i\in L^1((0,1))\},$$
endowed with the norm $||f||_{p,\rho_i}=\Big(\int_0^1|f|^p\rho_i dx\Big)^{\frac 1 p}.$

Naturally, the value at $0$ of a function $f\in L^p_{\rho_i}[0,1)$ carries no measure-theoretic information, but we prefer to include $0$ in the interval where  these spaces are defined to better highlight the geometric framework.

The space of functions on the entire metric tree, $L^p_{\rho}(L(G))$, is then defined as:
$$
L^p_{\rho}(L(G)) := \left\{(f_i)_{i\in I }\,:\, f_i\in L^p_{\rho_i}[0,1) \text{ and } \sum_{i\in I }||f_i||_{p,\rho_i}^p<\infty\right\}.
$$
This space is a Banach space when equipped with the norm:
$$\|f\|_{p,\rho}:=\left(\sum_{i\in I }||f_i||_{p,\rho_i}^p\right)^{\frac 1 p},\qquad  f=(f_i)_{i\in I } \in L^p_\rho(L(G)).$$

It will be useful to observe that  condition \eqref{int} yields  for any $f\in L^p_\rho(L(G))$ and for any $n\in \N$ 
    \begin{equation}\label{disnorm} ||f||^p_{p,\rho}=\sum_{i\in I }||f_i||_{p,\rho_i}^p \geq \sum_{i\in I}\sum_{j\in M_n(i)}||f_j||_{p,\rho_j}^p.\end{equation}

A straightforward proof gives that 
  \begin{align*}&F^p_\rho:=\oplus_{i\in I } L^p_{\rho_i}[0,1)=\\
  &=\left\{(f_i)_{i\in I }\,:\,  f_i\in L^p_{\rho_i}[0,1) \text{ and there exists a finite set } F\subseteq I  \text{ s.t. }  f_i=0 \text{ if } i \in I \setminus F\right\}\end{align*}
  is dense in $L^p_\rho(L(G))$, and, consequently, $\oplus_{i\in I } C_c(0,1)$  is dense in $L^p_{\rho}(L(G))$.

\subsection{Left translation semigroups on directed metric trees}

We now define the left translation semigroup $\{T_t\}_{t \geq 0}$ on $L^p_{\rho}(L(G))$. This family of operators describes the flow of a function's value along the directed paths of the tree. For any $s\in [0,1)$ and $t\geq 0$, let $n(t,s)$ be the integer part of $s+t$, i.e., the unique $n\in \N_0$ such that $n\leq s+t<n+1$.
Keeping in mind that $s\in [0,1)$, if $n_0\leq t<n_0+1$, then 
\begin{equation}\label{n}n(t,s)=\begin{cases} n_0 &\mbox{ if } 0\leq s<n_0+1-t\\
n_0+1 &\mbox{ if } n_0+1-t\leq s <1 1\end{cases}.\end{equation}

 For $f\in L^p_{\rho}(L(G))$ and $t\geq 0$ we define the left translation $T_t$ as 
$$(T_t f)(s):=\mathcal{A}^{n(t,s)} f(t+s-n(t,s)), \qquad s\in[0, 1).$$
More in detail, for $f=(f_i)_{i\in I} \in L^p_{\rho}(L(G))$ and $t \geq 0$,
the value of the translated function is given by:
$$
(T_t f)_i(s) = \sum_{j \in M_{n(t,s)}(i)} (\mathcal{A}^{n(t,s)})_{ij} \cdot f_j(t+s-n(t,s)), \qquad i\in I,\ s \in [0,1).
$$

The map defined in this way can intuitively be seen as a left translation with several contributions that are  described by the matrix $\mathcal{A}$. 

Note that if $t=n\in \mathds N$, then $n(t,s)=n$ for every $s\in [0,1)$ and therefore

\begin{equation*}%\label{t=n}
(T_n f)_i(s) = \sum_{j \in M_n(i)} (\mathcal{A}^{n})_{ij} \cdot f_j(s), \qquad s \in [0,1).
\end{equation*}

\begin{remark}\label{Ltr} If the tree is  a rooted sequence of connected edges of the following type:

\begin{center}\includegraphics[width=5cm]{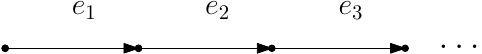}\end{center}
it holds that 
$$(\mathcal A)^{n}_{ij}\neq 0 \Leftrightarrow j=i+n,$$
hence for every $(f_i)_{i\in \N}\in L^p(L(G))$ and every  $i\in\N$
$$(T_tf)_i=f_{i+n(t,s)}(t+s-n(t,s)), \qquad s\in [0,1), t\geq 0.$$
Clearly $L^p_\rho(L(G))$ is isometric to $L^p_{\widetilde\rho}([0,+\infty))$,  
with 
\begin{equation}\label{wrho}\widetilde\rho(s)=\rho_i(s-i),\qquad    s\in [i,i+1),\  i\in\N,\end{equation} 
via the isometry
$$\Phi:(f_i)_{i\in I}\longmapsto f$$ 
where $f(u)=f_{[u]}(u-[u])$ for any $u\in [0,+\infty)$.

Under this isometry, the left translation semigroup on the directed metric tree  turns into the classical left translation  semigroup $(\widetilde T_t)_{t\geq 0}$ on $L^p_{\widetilde \rho}[0,+\infty)$. Indeed,
for every $f\in L^p_{\rho}(L(G))$ and $u\in [0,+\infty)$
\begin{align*}(\Phi( T_t((f_i)_{i\in\N}))(u)&= [T_t((f_i)_{i\in\N})]_{[u]}(u-[u])\\&=f_{[u]+n(t, u-[u])}(t+u-[u]-n(t,u-[u]))=\Phi((f_i)_{i\in\N})(t+u) \\&=\widetilde T_t (\Phi((f_i)_{i\in\N})(u).\end{align*}

Analogously, if the tree is an unrooted sequence of connected edges, then $L^p_\rho(L(G))$ is isometric to $L^p_{\widetilde\rho}(\R)$, where $\widetilde\rho$ is defined on $\R$ as in \eqref{wrho}, but with $i\in \mathds Z$. In this case,   the left translation semigroup on the directed tree  turns into the left translation semigroup on $L^p_{\widetilde\rho}(\R)$.
\end{remark}

\section{Strong continuity for the translation semigroup on directed metric trees}

\hspace{1em}In this section we provide a characterization of the weight $\rho$  such that the translation semigroup $\{T_t\}_{t\geq 0 }$ is  strongly continuous on the weighted space $L^p_\rho(L(G))$.

%\begin{lemma}\label{lema sucesiones}(\cite[Lemma 4.2]{karl-dimitris-hypercyclicity})
    %Let $I$ be a finite or countable set and $\rho=(\rho_j)_{j}\in\R_{+}^{I}$. Then 
    %\begin{align*}
        %&\inf_{\|x\|_1=1}\sum_{j\in I}|x_j|\rho_j=\inf_{j\in I}\rho_j,\\
        %&\inf_{\|x\|_1=1}\left(\sum_{j\in I}|x_j|^p\rho_j\right)=\left(\sum_{j\in I}\frac{1}{\rho_j^{1/(p-1)}}\right)^{-p/p^*},
    %\end{align*}
    %where, for $x\in \R^I$, $\|x\|_1=\sum_{j\in I}|x_j|$, $1<p<\infty$ and $p*=\frac{p}{p-1}$ is the conjugate exponent of $p$. 
%\end{lemma}

We begin with the following preparatory result that  can be proved by adapting the proof of  \cite[Lemma 4.3]{karl-dimitris-hypercyclicity}.

\begin{lemma}(see \cite[Lemma 4.3]{karl-dimitris-hypercyclicity})\label{lema sucesiones}
    Let $E\subseteq \R^d$ be a Lebesgue measurable set, $I$ a finite or countable set,  $\rho=(\rho_i)_{i\in I}$  a family  of positive measurable functions on $E$, and define the set
    $$\mathcal{V}=\left\{(v_i)_{i\in I}\,:\, v_i \text{ measurable on } E \text{ for every } i\in I, \ \ \sum_{i\in I}|v_i(s)|=1 \text{ for all }s\in E\right\}.$$
    Then, for  every $s\in E$
     \begin{align*}
        &%\label{f1} \inf_{(v_i)_i\in \mathcal{V}}\sum_{j\in I}|v_i(s)|\rho_i(s)=\inf_{i\in I}\rho_i(s), \\
        &\inf_{(v_i)_i\in \mathcal{V}}\left(\sum_{i\in I}|v_i(s)|^p\rho_i(s)\right)=\left(\sum_{i\in I}\frac{1}{\rho_i(s)^{1/(p-1)}}\right)^{1-p}, \qquad 1<p<\infty.%\label{fp}
\end{align*}

\end{lemma}

\begin{proposition}%\label{semigrupo traslacion}
    Let $\rho=(\rho_i)_{i\in I}$ be a weight on the directed metric graph $L(G)$ and let $M\geq 1$ and $w\in\mathds R$.
    \begin{enumerate}[(a)]
        \item The following assertions are equivalent:
        \begin{enumerate}[(i)]
            \item The family $\{T_t\}_{t\geq 0}$ is a strongly continuous semigroup on $L^1_{\rho}(L(G))$ such that $||T_t||\leq Me^{wt}$ for all $t\ge 0$;
            \item for all $i\in I $, $t\geq 0$, and  a.e. $s\in[0, 1)$ 
          $$\rho_i(s)\leq Me^{wt}\inf_{j\in M_{n(t,s)}(i)}\rho_j(s+t-n(t,s)).$$
          
        \end{enumerate}
        \item Let $1<p<\infty$.
        %and let $p^*=p/(p-1)$ be the conjugate exponent. 
        Then the following assertions are equivalent: 
        \begin{enumerate}[(i)]
           \item The family $\{T_t\}_{t\geq 0}$ is a strongly continuous semigroup on $L^p_{\rho}(L(G))$ such that $||T_t||\leq Me^{wt}$ for all $t\geq 0$;
            \item for all $i\in I $, $t\geq 0$,  and a.e. $s\in[0, 1)$
            $$\left(\sum_{j\in M_{n(t,s)}(i)}\frac{1}{\rho_j(s+t-n(t,s))^{1/(p-1)}}\right)^{p-1}\leq M^pe^{pwt}\frac{1}{\rho_i(s)}.$$
            
        \end{enumerate}
        \end{enumerate}
\end{proposition}
\begin{proof}
$(a)$  $(ii)\Rightarrow(i)$: We first show that  for every $t> 0$ the map $T_{t}$ is a continuous operator on $L^1_\rho(L(G))$ and $||T_t||\leq Me^{wt}$. 

Fix $t_0> 0$ and let  $n_0\in\N\cup\{0\}$ be such that $n_0\leq t_0< n_0+1$ and $f\in L^1_{\rho}(L(G))$.  Taking into account \eqref{disnorm} and \eqref{n},  we compute  the norm of $T_{t_0}f$ and apply  condition $(ii)$, so that:  
\begin{align*}
    &\|T_{t_0}f\|_{1,\rho}=\\ 
    &=\sum_{i\in I }\int_0^{n_0+1-t_0}\left|\left(\mathcal{A}^{n_0}f\right)_i(s+t_0-n_0)\right|\rho_i(s)ds\\
&\hspace{6em}+\sum_{i\in I }\int^1_{n_0+1-t_0}\left|\left(\mathcal{A}^{n_0+1}f\right)_i(s+t_0-n_0-1)\right|\rho_i(s)ds
   \\
   &=\sum_{i\in I }\int_0^{n_0+1-t_0}\left|\sum_{j\in M_{n_0}(i)}f_j(s+t_0-n_0)\right|\rho_i(s)ds\\
    &\hspace{6em}+\sum_{i\in I }\int_{n_0+1-t_0}^1\left|\sum_{j\in M_{n_0+1}(i)}f_j(s+t_0-n_0-1)\right|\rho_i(s)ds\\
    &\leq \sum_{i\in I }\int_0^{n_0+1-t_0}\sum_{j\in M_{n_0}(i)}|f_j(s+t_0-n_0)|\rho_i(s)ds\\
&\hspace{6em}+\sum_{i\in I }\int_{n_0+1-t_0}^1\sum_{j\in M_{n_0+1}(i)}|f_j(s+t_0-n_0-1)|\rho_i(s)
    ds\\
    %&\leq \sum_{i\in I }\int_0^{n_0+1-t_0}\sum_{j\in M_{n_0}(i)}|f_j(s+t_0-n_0)|Me^{wt_0}\rho_j(s+t_0-n_0)ds\\
%&\hspace{6em}+\sum_{i\in I }\int_{n_0+1-t_0}^1\sum_{j\in M_{n_0+1}(i)}|f_j(s+t_0-n_0-1)|Me^{wt_0}\rho_j(s+t_0-n_0-1)ds\\
    &\leq Me^{wt_0}\sum_{i\in I }\int_0^{n_0+1-t_0}|f_i(s+t_0-n_0)|\rho_i(s+t_0-n_0)ds\\
&\hspace{6em}+Me^{wt_0}\sum_{i\in I }\int_{n_0+1-t_0}^1|f_i(s+t_0-n_0-1)|\rho_i(s+t_0-n_0-1)ds\\
    &=Me^{wt_0} \sum_{i\in I }\int_{0}^{1}|f_i(s)|\rho_i(s)ds=Me^{wt_0}\|f\|_{1,\rho}.
\end{align*}
As a consequence, %$T_t$ is a continuous operator on $L^1_{\rho}(L(g))$ for each $t\geq 0$ and
the family $\{T_t\}_{t\geq 0}$ is locally equicontinuous on $L^1_\rho(L(G))$. 

\hspace{1em}  The strong continuity of the semigroup is established by showing that $$\lim_{t\to 0} \|T_t g- g\|_{1,\rho} = 0$$ holds for $g$ in a dense subset of $L^1_\rho(L(G))$, specifically for $g \in \oplus_{i\in I } C_c(0,1)$ (see Proposition \ref{strongly-condition}).

\hspace{1em} To this end, let $g\in \oplus_{i\in I } C_c(0,1)$, $I_g:=\{i\in I \hspace{0.3em}|\hspace{0.3em}g_i\neq  0 \}$ and $\epsilon>0$, and observe that $g_i$ is uniformly continuous for each $i\in I_g$. Since $I_g$ is finite,   there exists some $0<\overline t\leq 1$ such that  
$$\left|g_i(s+t)-g_i(s)\right|<\frac{\epsilon}{2\sum_{i\in I_g}\int_{0}^1\rho_i(s)ds}$$
for all $s\in[0, 1)$, $0<t\leq \overline t$ and $i\in I_g$.  Then, for $0<t\leq \overline t$, we obtain
\begin{align*}
    &\sum_{i\in I}\int_{0}^{1-t}\left|g_i(s+t)-g_i(s)\right|\rho_i(s)ds=\sum_{i\in I_g}\int_{0}^{1-t}\left|g_i(s+t)-g_i(s)\right|\rho_i(s)ds<\epsilon/2.
\end{align*}
On the other hand, there exists some $0<t_0\leq \overline t$ such that for all $t\leq t_0$,
$$Me^{wt}\sum_{i\in I}\int_{0}^t|g_i(s)|\rho_i(s)ds=Me^{wt}\sum_{i\in I_g}\int_{0}^t|g_i(s)|\rho_i(s)ds<\epsilon/4$$
and also 
$$\sum_{i\in I}\int_{1-t}^1|g_i(s)|\rho_i(s)ds=\sum_{i\in I_g}\int_{1-t}^1|g_i(s)|\rho_i(s)ds<\epsilon/4.$$
As a consequence, for $t\leq t_0$, we obtain: 

\begin{align*}
    &\|T_tg-g\|_{1,\rho}=\\
    &=\sum_{i\in I_g} \int_{0}^{1-t}\left|g_i(s+t)-g_i(s)\right|\rho_i(s)ds
    +\sum_{i\in I }\int_{1-t}^1|(\mathcal{A}g)_i(s+t-1)-g_i(s)|\rho_i(s)ds\\
    &\leq \sum_{i\in I}\int_{0}^{1-t}\left|g_i(s+t)-g_i(s)\right|\rho_i(s)ds
    \\
&\hspace{7em}+\sum_{i\in I }\int_{1-t}^1|(\mathcal{A}g)_i(s+t-1)|\rho_i(s)ds+\sum_{i\in I}\int_{1-t}^1|g_i(s)|\rho_i(s)ds
    \\
    &\leq\sum_{i\in I}\int_{0}^{1-t}\left|g_i(s+t)-g_i(s)\right|\rho_i(s)ds
    \\
&\hspace{7em}+\sum_{i\in I }\int_{1-t}^1 \left|\sum_{j\in M_1(i)}g_j(s+t-1)\right|\rho_i(s)ds+\sum_{i\in I}\int_{1-t}^1|g_i(s)|\rho_i(s)ds
    \\
    &\leq\sum_{i\in I}\int_{0}^{1-t}\left|g_i(s+t)-g_i(s)\right|\rho_i(s)ds
    \\
&\hspace{7em}+Me^{wt}\sum_{i\in I }\int_{1-t}^1 \sum_{j\in M_1(i)}|g_j(s+t-1)|\rho_j(s+t-1)ds\\
    &\hspace{7em}+ \sum_{i\in I_g}\int_{1-t}^1|g_i(s)|\rho_i(s)ds\\
    &\leq\sum_{i\in I}\int_{0}^{1-t}\left|g_i(s+t)-g_i(s)\right|\rho_i(s)ds
    \\
&\hspace{7em}+Me^{wt}\sum_{i\in I }\int_{0}^t |g_i(s)|\rho_i(s)ds\\
    &\hspace{7em}+ \sum_{i\in I_g}\int_{1-t}^1|g_i(s)|\rho_i(s)ds<\epsilon, \end{align*} 
 where we applied \eqref{disnorm} to get the last inequality.  
 
$(i)\Rightarrow(ii)$: Let us now assume that $\{T_t\}_{t\geq 0}$ is a strongly continuous semigroup in $L^1_\rho(L(G))$, such that $\|T_t\|\leq Me^{wt}$ for all $t\geq 0$. Assume that $(ii)$ does not hold.
Then there exist  $t_0>0$ and $i_0\in I $ such that the set
$$B:=\{s\in [0,1)\,:\, \rho_{i_0}(s)>Me^{w t_0}\inf_{j\in M_{n(t_0,s)}(i_0)}\rho_j(s+t_0-n(t_0, s)\}$$
has Lebesgue measure $\lambda(B)>0$.
Let $n_0\in\N\cup\{0\}$ such that $n_0\leq t_0<n_0+1$. Then 
$B=B_1\cup B_2$, where 
\begin{align*}&B_1:=\left\{s\in[0, n_0+1-t_0): \rho_{i_0}(s)>Me^{wt_0}\inf_{j\in M_{n_0}(i_0)}\rho_{j}(s+t_0-n_0)\right\},
\\&B_2:=\left\{s\in[n_0+1-t_0, 1): \rho_{i_0}(s)>Me^{wt_0}\inf_{j\in M_{n_0+1}(i_0)}\rho_{j}(s+t_0-n_0-1)\right\}.\end{align*}
Clearly either $\lambda(B_1)>0 $ or $\lambda(B_2)>0$.
Assume without loss of generality that $\lambda(B_1)>0$ (the argument for $\lambda(B_2)>0$ is entirely analogous).
Since 
$$B_1\subseteq \bigcup_{j\in M_{n_0}(i_0)}\{s\in [0,n_0+1-t_0]\,:\, \rho_{i_0}(s)>M e^{wt_0}\rho_j(s+t-n_0)\}$$
we get that there exists
$j_0\in M_{n_0}(i_0)$
such that the set 
$$B_1^{j_0}=\{s\in [0,n_0+1-t_
0]\,:\, \rho_{i_0}(s)>M e^{wt_0}\rho_{j_0}(s+t-n_0)\}$$
has strictly positive Lebesgue measure.
Define the function 
$f:=(f_i)_{i\in I }$ as
$$f_{i}(u):=
\left\{\begin{array}{@{}l@{}}
    \frac{1}{\rho_{j_0}(u)}\hspace{2em} \text{ if }i=j_{0}\text{ and }u\in B_1^{j_0}+t_0-n_0; \\
    0 \hspace{3.7em} \text{ otherwise. }
  \end{array}\right.\,$$
 Clearly $f\in L^1_\rho(L(G))$ and $$\|f\|_{1,\rho}=\sum_{i\in I }\int_{0}^1|f_i(s)|\rho_i(s)ds=\lambda(B_1^{j_0})>0.$$
 Nevertheless, 
\begin{align*}
    \|T_{t_0}f\|
    &=\sum_{i\in I }\int_0^{n_0+1-t_0}|(\mathcal{A}^{n_0}f)_i(s+t_0-n_0)|\rho_i(s)ds+\\
&+\sum_{i\in I }\int^1_{n_0+1-t_0}|(\mathcal{A}^{n_0+1}f)_i(s+t_0-n_0-1)|\rho_i(s)ds\\
    &=\sum_{i\in I }\int_0^{n_0+1-t_0}|\sum_{j\in M_{n_0}(i)}f_j(s+t_0-n_0)|\rho_i(s)ds\\
    &+\sum_{i\in I }\int^1_{n_0+1-t_0}|\sum_{j\in M_{n_0+1}(i)}f_j(s+t_0-n_0-1)|\rho_i(s)ds\\
    &\geq \sum_{i\in I }\int_0^{n_0+1-t_0}|\sum_{j\in M_{n_0}(i)}f_j(s+t_0-n_0)|\rho_i(s)ds =\int_0^{n_0+1-t_0}|f_{j_0}(s+t_0-n_0)|\rho_{i_0}(s)ds\\
&=\int_{B_1^{j_0}}|f_{j_0}(s+t_0-n_0)|\rho_{i_0}(s)ds\\
    & >Me^{wt_0}\int_{B_1^{j_0}}|f_{j_0}(s+t_0-n_0)|\rho_{j_0}(s+t_0-n_0)ds=Me^{wt_0}\lambda(B_1^{j_0})=Me^{wt_0}\|f\|_{1,\rho},
    \end{align*}
    which contradicts (i). Hence $(a.ii)$ must hold and the conclusion follows. \\
    
   %color{brass} 
   $(b)$ $(ii)\Rightarrow (i)$: Consider $t_0>0$, $n_0\in\N\cup\{0\}$ such that $n_0\leq t_0<n_0+1$ and $f\in L^p_{\rho}(L(G))$. Let $p^*$ be the conjugate exponent of $p$. By observing that $p^*/p=1/(p-1)$, by $(ii)$, Hölder's inequality and \eqref{disnorm}, we obtain
   \begin{align*}
    &\|T_{t_0}f\|_{p,\rho}^p=\\
       &=\sum_{i\in I }\int_{0}^{n_0+1-t_0}|(\mathcal{A}^{n_0}f)_i(s+t_0-n_0)|^p\rho_i(s)ds+\\
       &\hspace{6em}+\sum_{i\in I }\int_{n_0+1-t_0}^{1}|(\mathcal{A}^{n_0+1}f)_i(s+t_0-n_0-1)|^p\rho_i(s)ds\\
        &=\sum_{i\in I }\int_{0}^{n_0+1-t_0}\left|\sum_{j\in M_{n_0}(i)}f_j(s+t_0-n_0)\right|^p\rho_i(s)ds+\\
       &\hspace{6em}+\sum_{i\in I }\int_{n_0+1-t_0}^{1}\left|\sum_{j\in M_{n_0+1}(i)}f_j(s+t_0-n_0-1)\right|^p\rho_i(s)ds\\
       &=\sum_{i\in I }\int_{0}^{n_0+1-t_0}\left|\sum_{j\in M_{n_0}(i)}f_j(s+t_0-n_0)\frac{\rho_j(s+t_0-n_0)^{1/p}}{\rho_j(s+t_0-n_0)^{1/p}}\right|^p\rho_i(s)ds+\\
&\hspace{6em}+\sum_{i\in I }\int_{n_0+1-t_0}^{1}\left|\sum_{j\in M_{n_0+1}(i)}f_j(s+t_0-n_0-1)\frac{\rho_j(s+t_0-n_0-1)^{1/p}}{\rho_j(s+t_0-n_0-1)^{1/p}}\right|^p\rho_i(s)ds\\
       &\leq \sum_{i\in I }\int_{0}^{n_0+1-t_0}\left(\sum_{j\in M_{n_0}(i)}|f_j(s+t_0-n_0)|^p\rho_j(s+t_0-n_0)\right)\times\\
       &\hspace{14em}\times\left(\sum_{j\in M_{n_0}(i)}\frac{1}{\rho_j(s+t_0-n_0)^{1/(p-1)}}\right)^{p-1}\rho_i(s)ds\\
       &\hspace{3em}+\sum_{i\in I }\int_{n_0+1-t_0}^{1}\left(\sum_{j\in M_{n_0+1}(i)}|f_j(s+t_0-n_0-1)|^p\rho_j(s+t_0-n_0-1)\right)\times\\
       &\hspace{14em}\times\left(\sum_{j\in M_{n_0+1}(i)}\frac{1}{\rho_j(s+t_0-n_0-1)^{1/(p-1)}}\right)^{p-1}\rho_i(s)ds\\
       &\leq M^pe^{pwt_0}\sum_{i\in I }\int_0^{n_0+1-t_0}\sum_{j\in M_{n_0}(i)}|f_j(s+t_0-n_0)|^p\rho_j(s+t_0-n_0)ds+\\
&\hspace{4em}+M^pe^{pwt_0}\sum_{i\in I }\int_{n_0+1-t_0}^1\sum_{j\in M_{n_0+1}(j)}|f_j(s+t_0-n_0-1)|^p\rho_j(s+t_0-n_0-1)ds\\
       &\leq M^pe^{pwt_0}\|f\|_{p,\rho}^p.
   \end{align*}
   Hence, it follows that %$T_{t}$ is a continuous operator on $L^p_{\rho}(L(G))$ for each $t\geq 0$ and 
   the family $\{T_t\}_{t\geq 0}$ is locally equicontinuous. 

   \hspace{1em} Now we prove that  the semigroup $\{T_t\}_{t\geq 0}$ is strongly continuous. Let $g\in \oplus_{i\in I } C_c(0,1)$, $I_g:=\{i\in I \hspace{0.3em}|\hspace{0.3em}g_i\neq  0 \}$ and $\epsilon>0$.  As in case $(a)$, it is easy to see that there is some $t_0\in(0,1]$ such that for $0<t\leq t_0$:
\begin{align*}
    &\sum_{j\in I}\int_{0}^{1-t}\left|g_j(s+t)-g_j(s)\right|^p\rho_j(s)ds=\sum_{j\in I_g}\int_{0}^{1-t}\left|g_j(s+t)-g_j(s)\right|^p\rho_j(s)ds<\epsilon/2,\\
    &\sum_{j\in I}\int_{1-t}^1|g_j(s)|^p\rho_j(s)ds=\sum_{j\in I_g}\int_{1-t}^1|g_j(s)|^p\rho_j(s)ds<\frac{\epsilon}{2^{p+1}},\\
&M^pe^{pwt}\sum_{i\in I}\int_{0}^t|g_i(s)|^p\rho_i(s)ds=M^pe^{pwt}\sum_{i\in I_g}\int_{0}^t|g_i(s)|^p\rho_i(s)ds<\frac{\epsilon}{2^{p+1}}.\end{align*}
Therefore, by condition $(b.ii)$, for $t\leq t_0$ we have: 
\begin{align*}
    &\|T_tg-g\|_{p,\rho}^p=\\
    &=\sum_{i\in I }\int_{0}^{1-t}\left|g_i(s+t)-g_i(s)\right|^p\rho_i(s)ds
    +\sum_{i\in I }\int_{1-t}^1|(\mathcal{A}g)_i(s+t-1)-g(s)|^p\rho_i(s)ds\\
    &\leq \sum_{i\in I}\int_{0}^{1-t}\left|g_i(s+t)-g_i(s)\right|^p\rho_i(s)ds
    \\
    &\hspace{4em}+2^{p-1}\sum_{i\in I }\int_{1-t}^1|(\mathcal{A}g)_i(s+t-1)|^p\rho_i(s)ds+2^{p-1}\sum_{i\in I}\int_{1-t}^1|g_i(s)|^p\rho_i(s)ds
    \\
    &\leq\sum_{i\in I}\int_{0}^{1-t}\left|g_i(s+t)-g_i(s)\right|^p\rho_i(s)ds
    \\
    &\hspace{4em}+2^{p-1}\sum_{i\in I }\int_{1-t}^1 \left|\sum_{j\in M_1(i)}g_j(s+t-1)\frac{\rho_j(s+t-1)^{1/p}}{\rho_j(s+t-1)^{1/p}}\right|^p\rho_i(s)ds+2^{p-1}\sum_{i\in I_g}\int_{1-t}^1|g_i(s)|^p\rho_i(s)ds
    \\
&\leq\sum_{i\in I}\int_{0}^{1-t}\left|g_i(s+t)-g_i(s)\right|^p\rho_i(s)ds
    \\
    &\hspace{2em}+2^{p-1}\sum_{i\in I }\int_{1-t}^1\left( \sum_{j\in M_1(i)}|g_j(s+t-1)|^p\rho_j(s+t-1)\right)
    \left(\sum_{j\in M_{1}(i)}\frac{1}{\rho_j(s+t-1)^{1/(p-1)}}\right)^{p-1}\rho_i(s)ds\\
    &\hspace{2em} +2^{p-1}\sum_{i\in I_g}\int_{1-t}^1|g_i(s)|^p\rho_i(s)ds
    \\
    &\leq\sum_{i\in I}\int_{0}^{1-t}\left|g_i(s+t)-g_i(s)\right|^p\rho_i(s)ds
    \\
    &\hspace{4em}+M^pe^{pwt}2^{p-1}\sum_{i\in I }\int_{1-t}^1 \sum_{j\in M_1(i)}|g_j(s+t-1)|^p\rho_j(s+t-1)ds+2^{p-1}\sum_{i\in I_g}\int_{1-t}^1|g_i(s)|^p\rho_i(s)ds
    \\ &\leq\sum_{i\in I}\int_{0}^{1-t}\left|g_i(s+t)-g_i(s)\right|^p\rho_i(s)ds
    \\
    &\hspace{4em}+M^pe^{pwt}2^{p-1}\sum_{i\in I}\int_{0}^t|g_i(u)|^p\rho_i(u)du+2^{p-1}\sum_{i\in I}\int_{1-t}^1|g_i(s)|^p\rho_i(s)ds<\epsilon,
\end{align*}
where we applied \eqref{disnorm} to get the second-last inequality.
Hence the conclusion follows by Theorem \ref{strongly-condition}.

$(i)\Rightarrow(ii)$. Assume that $(b.ii)$ is not satisfied.  Then there exist some $t_0>0$, $i_0\in I $, 
such that the set
$$B:=\left \{s\in [0,1)\,:\, \left(\sum_{j\in M_{n(t_0,s)}(i_0)}\frac{1}{\rho_j(s+t_0-n(t_0,s))^{1/(p-1)}}\right)^{1-p}< \frac{\rho_{i_0}(s)}{Me^{wt_0}}\right\}$$
has strictly positive Lebesgue measure.
Let $n_0\in\N$ such that $n_0\leq t_0<n_0+1$. Then $B=B_1\cup B_2$ where 
\begin{align*}&B_1:=\left\{s\in[0, n_0+1-t_0): \left(\sum_{j\in M_{n_0}(i_0)}\frac{1}{\rho_j(s+t_0-n_0)^{1/(p-1)}}\right)^{1-p}< \frac{\rho_{i_0}(s)}{M^{p}e^{pwt_0}}\right\},\\
&B_2:=\left\{s\in[n_0+1-t_0, 1): \left(\sum_{j\in M_{n_0+1}(i_0)}\frac{1}{\rho_j(s+t_0-n_0-1)^{1/(p-1)}}\right)^{1-p}<\frac{\rho_{i_0}(s)}{M^{p}e^{pwt_0}}\right\}, \end{align*}
hence $B_1$ or $B_2$ have strictly positive Lebesgue measure. Assume without loss of generality that $\lambda(B_1)>0$. 
%Then we can define $r:[0, 1]\rightarrow \R_{+}$ such that 
%$$\left(\sum_{i\in M_{n_0}(j)}\frac{1}{\rho_i(s+t_0-n_0)^{1/(p-1)}}\right)^{1-p}+r(s+t_0-n_0)=\frac{\rho_{j_0}(s)}{M^{p}e^{pwt_0}}$$
%for all $s\in[0, n_0+1-t_0]$.

By Lemma \ref{lema sucesiones}, there exists a sequence $(v_j(\cdot +t_0-n_0))_{j\in M_{n_0}(i_0)}$ such that every $v_j (\cdot +t_0-n_0)$ is measurable on $B_1$, for every $s\in B_1$
$$\sum_{j\in M_{n_0}(i_0)}|v_j(s+t_0-n_0)|=1,$$
and also 
$$\sum_{j\in M_{n_0}(i_0)}|v_j(s+t_0-n_0)|^p\rho_j(s+t_0-n_0)<\frac{\rho_{i_0}(s)}{M^{p}e^{pwt_0}}.$$
Define the function $f$ on $L(G)$ such that 
$$f_{j}(t):=
\left\{\begin{array}{@{}l@{}}
    v_j(t)\hspace{3em} \text{ if }j\in M_{n_0}(i_{0})\text{ and }t\in B_1+t_0-n_0; \\
    0 \hspace{4.5em} \text{ otherwise. }
  \end{array}\right.\,$$

We have the following.

\begin{align*}
   &\|T_{t_0}f\|_{p,\rho}^p=\\
       &=\sum_{i\in I }\int_{0}^{n_0+1-t_0}|(\mathcal{A}^{n_0}f)_i(s+t_0-n_0)|^p\rho_i(s)ds+\\
       &\hspace{12em}+\sum_{i\in I }\int_{n_0+1-t_0}^{1}|(\mathcal{A}^{n_0+1}f)_i(s+t_0-n_0-1)|^p\rho_i(s)ds\\
        &=\sum_{i\in I }\int_{0}^{n_0+1-t_0}\left|\sum_{j\in M_{n_0}(i)}f_j(s+t_0-n_0)\right|^p\rho_i(s)ds+\\
       &\hspace{12em}+\sum_{i\in I }\int_{n_0+1-t_0}^{1}\left|\sum_{j\in M_{n_0+1}(i_0)}f_j(s+t_0-n_0-1)\right|^p\rho_i(s)ds\\
       &=\int_{0}^{n_0+1-t_0}\left|\sum_{j\in M_{n_0(i_0)}}f_j(s+t_0-n_0)\right|^p\rho_{i_0}(s)ds=\int_{B_1}\rho_{i_0}(s)ds.
       \end{align*}
       Nevertheless, 
\begin{align*}
 &\|f\|_{p,\rho}^p=\sum_{i\in I }\int_{0}^1|f_i(s)|^p\rho_i(s)ds=\int_{0}^1\sum_{i\in I }|f_i(s)|^p\rho_i(s)ds\\
    &=\int_{B_1}\sum_{j\in M_{n_0}(i_0)}|f_j(s+t_0-n_0)|^p\rho_j(s+t_0-n_0)ds\\
    &=\int_{B_1}\sum_{j\in M_{n_0}(i_0)}|v_j(s+t_0-n_0)|^p\rho_j(s+t_0-n_0)ds\\
    &<
    %\int_{B_1}\left(\sum_{j\in M_{n_0}(i)}\frac{1}{\rho_j(s+t_0-n_0)^{1/(p-1)}}\right)^{1-p}+r(s+t_0-n_0)ds\\
    \int_{B_1}\frac{\rho_{i_0}(s)}{M^{p}e^{pwt_0}}ds=\frac{\|T_{t_0}f\|^p}{M^pe^{pwt_0}}, 
\end{align*}
which contradicts the assumption, and therefore $(b.ii)$ must hold. 
\end{proof}

\begin{remark}
    When considering the specific tree structure described in Remark \ref{Ltr}, conditions (a)(ii) and (b)(ii) become: 
    $$\frac{\rho_i(s)}{\rho_{i+n(t,s)}(s+t-n(t,s))}\leq M^pe^{pw t}$$
    for all $i\in \N$, $t\geq 0$ and almost every $s\in [0,1)$. By translating this property to the  weight $\widetilde\rho$ on the half-line $[0, +\infty)$, we get that:
    $$\widetilde \rho(u)\leq M^pe^{pwt} \widetilde \rho(u+t), \qquad t\geq0,\ a.e.\ u\geq 0. $$
    This inequality is the precise condition that guarantees the strong continuity of the classical left translation semigroup on the space $L^p_{\widetilde \rho}([0, +\infty))$ (see \cite[Example 7.4]{Alfred}).
    
\end{remark}

In the following,  and in order to avoid some technical problems, we will consider a slightly stronger condition for the weight $\rho$.

\begin{definition} Let  $\rho=(\rho_i)_{i\in I}$ be a weight on the directed metric tree $L(G)$ such that  for every $i\in\ I$ and for every $s\in [0,1)$ it holds that $\rho_i(s)>0$.
    Then  $\rho$ is said to be  a $p$-admissible weight for $L(G)$ if there exists $M\geq 1 $ and $w\in\R$  such that for all $i\in I$,  $s\in[0, 1)$ and  $t\geq 0$:

          \begin{align}\rho_i(s)\leq Me^{wt}\inf_{j\in M_{n(t,s)}(i)}\rho_j(s+t-n(t,s)) \qquad &\mbox{ if } p=1\label{p=1}\\
          \left(\sum_{j\in M_{n(t,s)}(i)}\frac{1}{\rho_j(s+t-n(t,s))^{1/(p-1)}}\right)^{p-1}\leq M^pe^{pwt}\frac{1}{\rho_i(s)},\qquad  &\mbox{ if } 1<p<\infty.\label{1<p} 
          \end{align}           
\end{definition}

\begin{remark}\label{est} Observe that under conditions \eqref{p=1} and \eqref{1<p},  it holds that,  for every $i\in I$,  $s,t\geq 0$ such that $0\leq s+t<1$

\begin{align*}%\label{est1}
\rho_i(s)\leq Me^{w t} \rho_i(s+t)
\end{align*}
 and therefore, $0<\rho_i(0)\leq Me^{|w|} \rho_i(t)$ for every $t\in [0,1)$. Hence %since $\rho_i(1)>0$
we get that  $\inf_{s\in [0,1)}\rho_i(s)>0$. 
 Moreover, 
 \begin{align*}%\label{est2}
&\rho_i(s)\leq Me^{w (1-s)} \inf_{j\in M_1(i)}\rho_j(0) \qquad &\mbox{ if } p=1\\
&%\label{estp}
\rho_i(s) \leq Me^{w(1-s)}\left(\sum_{j\in M_1(i)}\rho_j(0)^{-1/(p-1)}\right)^{1-p} \qquad &\mbox{ if } 1<p<\infty.
\end{align*}

In any case, $\rho_i$ is bounded  and bounded away from $0$ on $[0,1)$.

\end{remark}

%\color{brass}
\section{Hypercyclicity for the left translation semigroup on directed metric trees}
\hspace{1em}In this section we   provide necessary and sufficient conditions for a $p$-admissible weight sequence $(\rho_i)_{i\in I }$ that ensure hypercyclicity for the translation semigroup on directed metric trees. We begin by observing that, 
    if the directed tree $G$ has a leaf, then the left translation semigroup cannot be hypercyclic in any admissible space $L^p_{\rho}(L(G))$. 
Indeed, let $f\in L^p_{\rho}(L(G))$ and $e=(v_{i_0}, v_{j_0})\in E$ be such that $v_{j_0}$ is a leaf of $G$ . Then it is clear that $(T_t f)_{e}(s)=0$ for all $t> 1$ and $0\leq s< 1$, and therefore the semigroup cannot be hypercyclic.

\begin{theorem}\label{weakly mixing translation}
Let $L(G)$ be the directed metric tree associated with $G=(V, E)$. Assume that $L(G)$ is rooted and that $G$ is without leaves.
\begin{enumerate}[(a)]
\item The following assertions are equivalent:
\begin{enumerate}[(i)]
\item the translation semigroup is hypercyclic on $L_{\rho}^1(L(G))$;
\item the translation semigroup is weakly mixing on $L_{\rho}^1(L(G))$;
\item there exists some increasing sequence $(n_h)_{h\in\N}\in\N^{\N}$ such that for all $i\in I$:
\begin{equation*}%\label{condicion weakly mixing}
\lim_{h\rightarrow\infty}\inf_{j\in M_{n_h}(i)}\inf_{s\in[0, 1)}\rho_j(s)=0.
\end{equation*}
\end{enumerate}
\item If $1<p<\infty$, the following assertions are equivalent:
\begin{enumerate}[(i)]
\item the translation semigroup is hypercyclic on $L^p_{\rho}(L(G))$;
 \item the translation semigroup is weakly mixing on $L^p_{\rho}(L(G))$;
\item there exists some increasing sequence $(n_h)_{h\in\N}\in\N^{\N}$ such that for all $i\in I$:
$$\lim_{h\to\infty}\sup_{s\in[0, 1)}\left(\sum_{j\in M_{n_h}(i)}\frac{1}{\rho_j(s)^{1/(p-1)}}\right)=\infty.$$

\end{enumerate}
\end{enumerate}
\end{theorem}
\begin{proof}
We will only provide the proof of assertion (b), as the proof of (a) is analogous.

The equivalence of $(ii)$ and $(i)$ follows  from the fact that $F_\rho^p$ is a dense subspace such that the orbit of each $f\in F^p_\rho$ is bounded (see \cite[Theorem 2.48]{Alfred}). 

Along the proof, it will be useful to observe that the limit in $(iii)$ can be reformulated as
$$\lim_{h\rightarrow \infty}\inf_{s\in [0, 1)}\left(\sum_{j\in M_{n_h}(i)}\frac{1}{\rho_j(s)^{1/(p-1)}}\right)^{-1}=0.$$

We will prove that $(iii)\Rightarrow(i)$, by showing that  given $\epsilon>0$, and $f_1, f_2\in F^p_\rho$, there is some $g\in L_{\rho}^p(L(G))$ and $t>0$ such that
$$\|f_1-g\|_{p,\rho} <\epsilon\quad\text{ and }\quad T_{t}g=f_2,$$
which by density of $F^p_\rho$ implies that $\{T_t\}_{t\geq 0}$ is topologically transitive and consequently the translation semigroup will be hypercyclic.

Observe that since $f_1\in F^p$, there is some $N_0\in \N$ such that
$$T_n f_1=0$$for all $n\geq N_0$. Now, take $M\geq 1$ and $w\in\R$ such that $\rho$ satisfies the condition \eqref{1<p}. Furthermore, define $I_1:=\{i\in I : (f_{1})_i\neq  0\}$, $I_2:=\{i\in I : (f_{2})_i  \neq  0\}$ and the constants:
$$C_1:=\min_{i\in I_{2}}{{\inf}}_{s\in [0, 1)}\rho_i(s)>0,\quad C_2:=\max_{i\in I_{2}}{{\sup}}_{s\in [0, 1)}\rho_i(s)>0.$$
Given $\delta>0$,  by $(iii)$ we have that there is some $n_h\in\N$ such that $n_h>N_0$ and for all $i\in I_2$ there exist some $s_{i, n_h}\in [0, 1)$ with
\begin{equation}\label{ec-b1}
 \left(\sum_{j\in M_{n_h(i)}}\frac{1}{\rho_j(s_{i, n_h})^{1/(p-1)}}\right)^{1-p}<\delta.
\end{equation}

By Lemma \ref{lema sucesiones}, for each $i\in I_2$  there exists a sequence $(\nu_{j}^i)_{j\in M_{n_h}(i)}$, $\nu_j^i\geq 0$ such that
\begin{equation}\label{ec-b2}
 \sum_{j\in M_{n_h}(i)}\nu_j^i=1
\end{equation}
and also
\begin{equation}\label{ec-b3}
\sum_{j\in M_{n_h}(i)}|\nu_{j}^i|^p\rho_j(s_{i, n_h})<\delta.
\end{equation}
Define the function $g\in L^{p}_{\rho}(L(G))$ such that
$$g_i(s):=\left\{\begin{array}{@{}l@{}}
(f_1)_i(s)\hspace{4em} \text{ if }\hspace{0.1em}i\in I_1;\\
(f_2)_l(s)\nu_{i}^l \hspace{2em} \text{ if } i \in M_{n_h}(l)\text{ for some }l\in I_2;\\
0\hspace{6.1em}\text{ otherwise.}
\end{array}\right.\,$$
Since $n_h>N_{0}$ we have by \eqref{ec-b2} that if $i\in I_2$ then
$$(T_{n_h}g)_i(s)=\sum_{j\in M_{n_h}(i)}(f_2)_i(s)\nu_j^i=(f_2)_i(s).$$
On the other hand, if $i\notin I_2$ then $(T_{n_h}g)_i=0$, and consequently $T_{n_h}g=f_2$.

Now we prove that $\|f_1-g\|_{p,\rho}<\epsilon$. 
By equation \eqref{ec-b3} we have:
\begin{align*}
&\|f_1-g\|_{p, \rho}^p=\sum_{i\in I }\int_{0}^1|(f_1)_i(s)-g_i(s)|^p\rho_i(s)ds\\
&=\sum_{l\in I_2}\sum_{j\in M_{n_h}(l)}\int_{0}^1|(f_2)_l(s)\nu_j^l|^p\rho_j(s)ds\\
&=\sum_{l\in I_2}\int_{0}^{1}|(f_2)_l(s)|^p\sum_{j\in M_{n_h}(l)}|\nu_{j}^l|^p\rho_j(s)ds\\
&\leq \sum_{l\in I_2}\int_{0}^{1}|(f_2)_l(s)|^p\sum_{j\in M_{n_h}(l)}|\nu_{j}^l|^p\frac{C_2}{C_1}\rho_j(s_{l, n_h})ds\leq \frac{C_2}{C_1}\delta\sum_{l\in I_2}\int_{0}^1|(f_2)_l(s)|^pds\\
&\leq \frac{C_2}{C_1^2}\delta\sum_{l\in I_2}\int_{0}^1|(f_2)_ l(s)|^p\rho_l(s)ds= \frac{C_2}{C_1^2}\delta \|f_2\|^p_{p,\rho}.
\end{align*}
We obtain the assertion by choosing $\delta=\frac{C_1^2}{C_2}{\|f_2\|_{p,\rho}}^{-p}\epsilon^p$.

\hspace{1em}We will now prove that $(ii)\Rightarrow (iii)$. Assume that the translation semigroup is weakly mixing but $\rho$ does not satisfy condition $(iii)$. By Theorem \ref{weaklymixing-mixingdiscretization} every autonomous discretization is weakly mixing, which holds in particular for the sequence $(T_n)_{n\in\N}$. As a consequence, the sequence $(T_n)_n$ is hereditarily transitive (see e.g. \cite[Exercise 1.6.4]{Alfred}), namely there exists an increasing sequence of natural numbers $(n_k)_{k\in\N}$ such that $(T_{m_k})_{k\in\N}$ is transitive for every subsequence $(m_k)_{k\in\N}$ of $(n_k)_{k\in\N}$.
 
Assume that there exists some $i_0\in I $
such that 
$$\lim_{k\rightarrow \infty}\inf_{s\in [0, 1)}\left(\sum_{j\in M_{n_k}(i_0)}\frac{1}{\rho_j(s)^{1/(p-1)}}\right)^{-1}\neq 0.$$

Then there exists a subsequence $(n_{k( \xi)})_{\xi\in\N}$ and $\varepsilon_0>0$ such that
$$\inf_{s\in[0, 1)}\left(\sum_{j\in M_{n_{k( \xi)}}(i_0)}\frac{1}{\rho_j(s)^{1/(p-1)}}\right)^{1-p}>\varepsilon_0.$$
Taking Remark \ref{est} into account, we get that  there exists some $K\geq 0$ such that
 \begin{equation}\label{ec-b5} \rho_{i_0}(s)\leq K\left(\sum_{j\in M_{n_{k(\xi)}}(i_0)}\frac{1}{\rho_j(s)^{1/(p-1)}}\right)^{1-p}
\end{equation}
for all $\xi\in\N$ and $s\in[0, 1)$.
Now, consider $g\in L_{\rho}^p(L(G))$ defined as
$$g_i(s):=\left\{\begin{array}{@{}l@{}}
 \frac{2^{1-1/p}}{\rho_{i_0}(s)^{1/p}}\hspace{2em} \text{ if }\hspace{0.1em}i=i_0; \\
0 \hspace{4.6em} \text{ otherwise. }
 \end{array}\right.\,$$
By equation \eqref{ec-b5} and Hölder's inequality, for any $f\in L_{\rho}^p(L(G))$ such that $\|f\|_{p,\rho}^p<\frac{1}{2K}$ and any $\xi\in\N$ we obtain:
\begin{align*}
&\|T_{n_{k(\xi)}}f-g\|_{p,\rho}^p= \sum_{i\in I }\int_{0}^{1}|(\mathcal{A}^{n_{k( \xi)}}f)_i(s)-g_i(s)|^p\rho_i(s)ds\\
&\geq \int_{0}^{1}|(\mathcal{A}^{{n_{k(\xi)}}}f)_{i_0}(s)-g_{i_0}(s)|^p\rho_{i_0}(s)ds\geq \int_{0}^{1}\left(2^{1-p}|g_{i_0}(s)|^p-|(\mathcal{A}^{{n_{k(\xi)}}}f)_{i_0}(s)|^p\right)\rho_{i_0}(s)ds\\
%=\int_{0}^{1}2^{1-p}|g_{i_0}(s)|^p\rho_{i_0}(s)ds-\int_{0}^{1}|(\mathcal{A}^{n_{k,\\xi}}f)_{i_0}(s)|^p\rho_{i_0}(s)ds\\
&=1-\int_{0}^{1}|(\mathcal{A}^{{n_{k(\xi)}}}f)_{i_0}(s)|^p\rho_{i_0}(s)ds=1-\int_{0}^{1}\left|\sum_{j\in M_{{n_{k(\xi)}}}(i_0)}f_j(s)\right|^p\rho_{i_0}(s)ds \\
&\geq 1-\int_{0}^{1}\left(\sum_{j\in M_{{n_{k(\xi)}}}(i_0)}|f_j(s)|^p\rho_j(s)\right)\left(\sum_{j\in M_{{n_{k(\xi)}}}(i_0)}\frac{1}{\rho_j(s)^{p*/p}}\right)^{p-1}\rho_{i_0}(s)ds\\
&\geq 1-K\int_{0}^{1}\sum_{j\in M_{{n_{k(\xi)}}}(i_0)}|f_j(s)|^p\rho_j(s)ds\\
&\geq 1-K\|f\|_{p,\rho}^p\geq 1/2.
\end{align*}
\hspace{1em}Hence, the sequence of operators $(T_{n_{k(\xi)}})_{\xi\in\N}$ is not transitive, which is a contradiction. Therefore, if the translation semigroup is weakly mixing,  condition $(iii)$ must hold.
\end{proof}

\begin{remark}
    A slight modification of the  proof of Theorem \ref{weakly mixing translation} allow to show  that, either on $L_{\rho}^1(L(G))$ or on $L_{\rho}^p(L(G))$, the translation semigroup is mixing if and only if condition $(iii)$ on the weight of the space holds for the full sequence $(n
    )_{n\in\mathds{N}}$.
\end{remark}

\hspace{1em} We will now characterize hypercyclicity for the unrooted case. In order to do so, given a directed metric tree $L(G)$, where $G=(V, E)$ is a directed tree, and $E=(e_i)_{i\in I}$, we define $K_n(j)\in I$ as the index of the unique edge such that $j\in M_n(K_n(j))$.  The edge $e_{K_n(j)}$ can be considered as an ancestor of $e_j$. Moreover,  we define the set
$$G(j)=\bigcup_{n=1}^{\infty}M_n(K_n(j)),$$
and observe that if $j_1\in G(j)$ then $G(j_1)=G(j)$.

Furthermore, we say that $j_1\sim_G j_2$ if $j_1\in G(j_2)$, so that $\sim_G$ defines an equivalence relation on $E$.

\begin{theorem}\label{TH3}
    Let $G$ be an \textit{unrooted} directed tree  without leaves and let $L(G)$ be the directed metric tree associated with $G$.
    
    \begin{enumerate}[(a)]
        \item The following assertions are equivalent: 
        \begin{enumerate}[(i)]
            \item the left translation semigroup is hypercyclic on $L_\rho^1(L(G))$;
            \item the left translation semigroup is weakly mixing on $L_{\rho}^1(L(G))$;
            \item there exists  an  increasing sequence $(n_{h})_{h}\in\N^{\N}$ such that for all $i\in I $
            \begin{align*}
&\lim_{h\rightarrow\infty}\inf_{s\in[0, 1)}\inf_{j\in M_{n_h}(i)}\rho_j(s)=0,\\ &\hspace{3em}\text{and }\\
&\lim_{h\rightarrow\infty}\inf_{s\in[0, 1)}\min\left\{\rho_{K_{n_{h}}(i)}(s),\inf_{j\in M_{n_h}(K_{n_h}(i))}\rho_j(s)\right\}=0.
            \end{align*}
        \end{enumerate}
       \item For $1<p<\infty$, the following assertions are equivalent: 
        \begin{enumerate}[(i)]
            \item the left translation semigroup is hypercyclic on $L_\rho^p(L(G))$;
            \item the left  translation semigroup is weakly mixing on $L_\rho^p(L(G))$;
            \item there exists an increasing sequence  $(n_h)_{h}\in\N^{\N}$ such that for all $i\in I $
            \begin{align*}
    &\lim_{h\to\infty}\sup_{s\in[0, 1)}\left(\sum_{j\in M_{n_h}(i)}\frac{1}{\rho_j(s)^{1/(p-1)}}\right)=\infty%\label{unr1}
    \\ &\hspace{4em}\text{and }\notag \\
    &\lim_{h\to\infty}\sup_{s\in [0,1)}\left(\frac{1}{\rho_{K_{n_h}(i)}(s)^{1/(p-1)}}+ \sum_{j\in M_{n_h}(K_{n_h}(i))}\frac{1}{\rho_j(s)^{1/(p-1)}}\right)=\infty%\label{unr3}
            \end{align*}
        \end{enumerate}
    \end{enumerate}

\end{theorem}

\begin{proof}
    Since the proofs for both cases are similar, we only present the proof of  case $(b)$. We turn first our attention to the implication  $(i) \Rightarrow (iii)$. 

We will prove that given a finite subset $F\subseteq I$,  for every $\epsilon >0$  and  every $N\in\N$ there exists $n\geq N$ such that for every $i\in F$:

\begin{align*}
            &\sup_{s\in[0, 1)}\left(\sum_{j\in M_{n(i)}}\frac{1}{\rho_j(s)^{1/(p-1)}}\right)>\frac{1}{\varepsilon},%\label{unr11}
            \\
            &\sup_{s\in [0,1)}\left(\frac{1}{\rho_{K_{n}(i)}(s)^{\frac{1}{p-1}}}
            + \sum_{j\in M_{n}(K_{n}(i))}\frac{1}{\rho_j(s)^{1/(p-1)}}\right)>\frac{1}{\epsilon}.%\label{unr21}
            \end{align*}
            
The assertion will follow by taking an increasing  sequence $(F_h)_{h\in\N}$ of finite subsets of $I$ such that $\bigcup_h F_h=I$ and  a sequence $(\epsilon_h)_{h\in\N}$ of positive numbers tending to $0$. Then, we can construct an increasing sequence $(n_h)_{h\in\mathds N}$ of natural numbers such that 

\begin{align*}
  \forall h\in\N\  \forall i\in F_h\qquad       &\sup_{s\in[0, 1)}\left(\sum_{j\in M_{n_h}(i)}\frac{1}{\rho_j(s)^{1/(p-1)}}\right)>\frac{1} {\epsilon_h},\\ &\sup_{s\in [0,1)}\left(\frac{1}{\rho_{K_{n_h}(i)}(s)^{1/(p-1)}}+ \sum_{j\in M_{n_h}(K_{n_h}(i))}\frac{1}{\rho_j(s)^{1/(p-1)}}\right)>\frac{1}{\epsilon_h}.
            \end{align*}
If $i\in I$, then $i\in F_h$ for any $h\geq h_0$ for some $h_0\in \N$, and therefore 
\begin{align*}\forall h\in\N, h\geq h_0:\            &\sup_{s\in[0, 1)}\left(\sum_{j\in M_{n_h}(i)}\frac{1}{\rho_j(s)^{1/(p-1)}}\right)>\frac{1}{\varepsilon_h},\\ &\sup_{s\in [0,1)}\left(\frac{1}{\rho_{K_{n_h}(i)}(s)^{1/(p-1)}}+ \sum_{j\in M_{n_h}(K_{n_h}(i))}\frac{1}{\rho_j(s)^{1/(p-1)}}\right)>\frac{1}{\varepsilon_h}
            \end{align*}
and the assertion follows.

So, consider a  finite subset $F\subset I$ and a subset $H\subset F$ such that for all $i\in F$ there is some $i_1\in H$ such that  $i_1\sim_G i $  and if $i_1, i _2\in H$ then $i_1\not\sim_G i _2$ (namely $H$ is a set of representatives with respect to the equivalence relation $\sim_G$ in $F$). Define the functions 
$$(f_F)_i(s):=\left\{\begin{array}{ll}
     &1 \hspace{2em}\text{ if }i\in F,\\
     &0 \hspace{2em} \text{ otherwise}
\end{array}\right.\quad \text{ for all }s\in[0, 1),$$
$$(f_H)_i(s):=\left\{\begin{array}{ll}
     &1 \hspace{2em}\text{ if }i\in H,\\
     &0 \hspace{2em} \text{ otherwise} 
\end{array}\right. \quad \text{ for all }s\in[0, 1).$$
Since the semigroup $\{T_t\}_{t \geq 0}$ is hypercyclic, by  Theorem \ref{hypercyclic-autonomous}, it follows that the operator $T_1$ is also hypercyclic. Hence, 
given $\varepsilon\in ]0, 4^{\frac 1 p}\min_{i\in F}\inf\rho_i[$ and $N\in\mathds N$,  there exist  $f\in L_{\rho}^p(L(G))$ and $n\geq N$ such that
\begin{align}\label{ec-3.3-1}
    &\|f-f_H\|_{p,\rho}<\frac{\epsilon}{4^{1/p}}, 
%\label{ec-3.3-2}
    &\|T_{n}f-f_F\|_{p,\rho}<\frac{\epsilon}{4^{1/p}},
\end{align}
and we can choose $n$ large enough so that 
\begin{equation}\label{ec-3.3-3}
    H\cap M_n(F)=\emptyset, \qquad K_n(H)\cap F=\emptyset
\end{equation}

and 
\begin{equation}\label{ec-3.3-3a}
i,j\in F,\ i\sim_G j\ \Rightarrow\ K_n(i)=K_n(j).
\end{equation}

\emph{Claim 1}: for all $i\in F$
\qquad 
$$\sup_{s\in[0, 1)}\left(\sum_{j\in M_{n}(i)}\frac{1}{\rho_j(s)^{1/(p-1)}}\right)\geq\left(\frac{4 }{\varepsilon^p\sup \rho_i}\right)^\frac{1}{p-1} \left( (\inf \rho_i)^{\frac 1 p} -\frac{\epsilon}{4^{\frac 1 p}}\right)^\frac{p-1}{p}$$ 

%$\inf_{s\in[0, 1)}\left(\sum_{j\in M_{n(i)}}\frac{1}{\rho_j(s)^{1/(p-1)}}\right)^{-1}\leq \left(\frac{\varepsilon}{1-\varepsilon}\right)^p \sup_{s\in[0, 1)}\rho_i(s)$.\\

Indeed, by  \eqref{ec-3.3-1} and \eqref{ec-3.3-3},   for each $i\in F$ 
$$\sum_{j\in M_{n}(i)}\int_0^1|f_j(s)|^p\rho_j(s)ds<\frac{\epsilon^p}{4},$$
and %by equation \eqref{ec-3.3-2}and by equation \eqref{ec-3.3-2}
\begin{align*} &\left(\int_0^1 |(T_nf)_i(s)-1|^p\rho_i(s)ds\right)^{\frac 1 p} \leq \|T_{n}f-f_F\|_{p,\rho}<\frac{\epsilon}{4^{1/p}}\end{align*}
and therefore 
$$\left(\int_{0}^1|(T_nf)_i(s)|^p\rho_i(s)ds\right)^\frac 1 p>\left(\int_0^1\rho_i(s)ds\right)^\frac 1 p-\frac{\epsilon}{4^{1/p}} \geq (\inf\rho_i)^\frac 1 p - \frac{\epsilon}{4^{\frac 1 p}}.$$
Hence, we get :
\begin{align*}
    &\left(\inf\rho_i\right)^\frac 1 p - \frac{\epsilon}{4^{\frac 1 p}}<\\
    &\left(\int_{0}^1|(T_nf)_i(s)|^p\rho_i(s)ds\right)^{\frac 1 p}=\left(\int_{0}^1\left|\sum_{j\in M_{n(i)}}f_j(s)\right|^p\rho_i(s)ds\right)^{\frac 1 p}\\
    &\leq \left(\int_{0}^1\left(\sum_{j\in M_{n(i)}}|f_j(s)|^p\rho_j(s)\right)\left(\sum_{j\in M_{n(i)}}\frac{1}{\rho_j(s)^{1/(p-1)}}\right)^{p-1}\rho_i(s)ds\right)^\frac 1 p\\
    &\leq \left(\sup\rho_i\right)^\frac 1 p \cdot \sup_{s\in[0, 1)}\left(\sum_{j\in M_{n}(i)}\frac{1}{\rho_j(s)^{1/(p-1)}}\right)^\frac{p-1}{p}\left(\sum_{j\in M_{n}(i)}\int_{0}^1|f_j(s)|^p\rho_j(s)ds\right)^\frac 1 p\\
    &\leq \frac{\epsilon}{4^{1/p}} \left(\sup\rho_i\right)^\frac 1 p\cdot \sup_{s\in[0, 1)}\left(\sum_{j\in M_{n}(i)}\frac{1}{\rho_j(s)^{1/(p-1)}}\right)^\frac{p-1}{p}.
\end{align*}

Consequently: 

$$\sup_{s\in[0, 1)}\left(\sum_{j\in M_{n}(i)}\frac{1}{\rho_j(s)^{1/(p-1)}}\right)\geq\left(\frac{4 }{\epsilon^p \sup \rho_i}\right)^\frac{1}{p-1} \left( (\inf \rho_i)^{\frac 1 p} -\frac{\epsilon}{4^{\frac 1 p}}\right)^\frac{p-1}{p} $$

\emph{Claim 2:} For each $i\in F$ 
%$$\left(\sup_{s\in[0, 1)}\frac{1}{\rho_{K_{n}(i)}(s)^{1/(p-1)}}\right)+\inf_{s\in[0, 1)}\sum_{j\in M_{n}(K_{n}(i))}\frac{1}{\rho_j(s)^{1/(p-1)}}\geq \frac{1}{(2\epsilon)^{p/(p-1)}}$$ 

$$\sup_{s \in [0, 1)}\left(\frac{1}{\rho_{K_{n}(i)}(s)^{1/(p-1)}}+\sum_{j\in M_{n}(K_{n}(i))} \frac{1}{\rho_j(s)^{1/(p-1)}}\right)\geq  \frac{1}{(2\epsilon)^{p/(p-1)}}$$

Observe that, by the definition of $H$ and \eqref{ec-3.3-3a}, it is enough to prove Claim 2 for any $i\in H$. To this aim,  fix $i\in H$  and consider the function $g$ defined by 
$$g_j(s)=\left\{\begin{array}{ll}
 &f_j(s)-\delta_{ij}\hspace{2em}\text{if }\hspace{0.2em}j\in M_n(K_n(i))\\
     &0\hspace{3.5em}\text{ otherwise}
\end{array}\right., \qquad s\in [0,1). $$
 By  \eqref{ec-3.3-1}
we have 
\begin{align*}%\label{ec-3.3-4}
    &\sum_{j\in M_{n}(K_{n}(i))}\int_{0}^1|g_j(s)|^p\rho_j(s)ds \\=&\int_0^1|f_i(s)-1|^p\rho_i(s)ds+  \sum_{j\in M_{n}(K_{n}(i))\setminus\{i\}}\int_{0}^1|f_j(s)|^p\rho_j(s)ds \notag\\
    \leq& ||f-f_H||^p<\frac{\epsilon^p}{4}\notag, 
\end{align*}
by observing that $\left(M_{n}(K_{n}(i))\setminus\{i\}\right)\cap H=\emptyset$. 
Therefore
\begin{align*}
    &\lambda\left(\left\{s\in[0, 1) \, : \, \sum_{j\in M_{n}(K_{n}(i))}|g_j(s)|^p\rho_j(s)\geq \epsilon^p\right\}\right)\\
    &\leq\frac{1}{\epsilon^p}\int_{0}^1\sum_{j\in M_{n}(K_{n}(i))}|g_j(s)|^p\rho_j(s)ds<\frac{1}{4}.
\end{align*}
Hence there exists some set $E_1\subseteq [0,1)$ with $\lambda(E_1)>3/4$ such that for all $s\in E_1$
$$\sum_{j\in M_{n}(K_{n}(i))}|g_j(s)|^p\rho_j(s)<\epsilon^p.$$

By  \eqref{ec-3.3-1} and \eqref{ec-3.3-3}, for each  $i\in H$ we have 
\begin{align*}%\label{ec-3.3-5}
    &\int_{0}^1|(T_{n}f)_{K_{n}(i)}|^p\rho_{K_{n}(i)}(s)ds=\int_0^1 \left|\sum_{j\in M_{n}(K_n(i))}f_j(s)\right|^p\rho_{K_{n}(i)}(s)ds\\
=&\int_{0}^1\left|\left(1+\sum_{j\in M_{n}(K_{n}(i))}g_j(s)\right)\right|^p\rho_{K_{n}(i)}(s)ds<\frac{\epsilon^p}{4},
\end{align*}
and, again we  conclude that there exists a set $E_2\subseteq [0,1)$ such that $\lambda(E_2)>3/4$ and for all $s\in E_2$
\begin{align*}
    \rho_{K_{n}(i)}(s)\left|\left(1+\sum_{j\in M_{n}(K_{n}(i))}g_j(s)\right)\right|^pds \leq \epsilon^p.
\end{align*}
Define now $E:=E_1\cap E_2$. Since $\lambda(E_1), \lambda(E_2)>3/4$, it is clear that $\lambda(E_1\cap E_2)>0$. Let $s_0\in E$. 
If $\rho_{K_n(i)}(s_0)\leq (2\epsilon)^p$, then Claim 2 is satisfied. If $ \rho_{K_n(i)}(s_0)\geq  (2\epsilon)^p$, then 
$$\left\vert 1+\sum_{j\in M_{n}(K_{n}(i))}g_j(s_0)\right|^p \leq \frac{1}{2^p}.$$
This yields that 
$$\sum_{j\in M_{n}(K_{n}(i))}|g_j(s_0)|\geq
\left\vert\sum_{j\in M_{n}(K_{n}(i))}g_j(s_0)\right\vert \geq \frac 1 2.$$
For any $j\in M_{n}(K_{n}(i))$,  consider the functions 
$$v_j(s_0)=\frac{g_j(s_0)}{\sum_{j\in M_{n}(K_{n}(i))}|g_j(s_0)|}\chi_{E\cap\{\rho_{K_n(i)}\geq (2\epsilon)^p\}}, \qquad s_0\in [0,1).$$
Every $v_j$ is measurable on $E$ and $\sum_{j\in M_{n}(K_{n}(i))}|v_j(s_0)|=1$ for any $s\in E\cap\{\rho_{K_n(i)}\geq (2\epsilon)^p\}$. Therefore, by Lemma \ref{lema sucesiones},
$$\left(\sum_{j\in M_{n}(K_{n}(i))} \frac{1}{\rho_j(s_0)^{1/(p-1)}}\right)^{1-p} \leq \sum_{j\in M_{n}(K_{n}(i))} |v_j(s_0)|^p\rho_j(s_0) <(2\varepsilon)^p$$
for any $s_0\in {E\cap\{\rho_{K_n(i)}\geq (2\epsilon)^p\}}.$

Hence for any $s_0\in E$

$$\sum_{j\in M_{n}(K_{n}(i))} \frac{1}{\rho_j(s_0)^{1/(p-1)}}\geq  \frac{1}{(2\epsilon)^{p/(p-1)}},$$
and Claim 2 follows.

To prove the implication $(iii)\Rightarrow (i)$, it is enough to show that, given $f_1, f_2\in F^p_{\rho}$ and $\epsilon >0$, there is some $g\in L_{\rho}^p(L(G))$ and $\overline t>0$ such that 
  $$\|f_1-g\|_{p,\rho}<\epsilon\quad\text{ and }\quad \|T_{\overline{t}}g-f_2\|_{p,\rho}<\epsilon.$$
 
 Without loss of generality, we can assume that, for $k=1,2$,  if  $(f_k)_i\neq  0$ for some $i\in I$,  then  $(f_k)_h=0$ for all $h\sim_G i,\ h\neq i$.
 Indeed,  any $f\in  F^p_{\rho}$ is a linear combination of functions satisfying this condition.
  
 %Let  $M\geq 1$ and $w\in\R$ such that $\rho$ verifies the admissibility condition in \eqref{1<p}.
 Define $I_1:=\{i\in I : (f_1)_i\neq  0\}$, $I_2:=\{i\in I : (f_2)_i\neq  0\}$ and the constants: 
\begin{align*}&C_1:=\min_{i\in I_{1}\cup I_2}\operatorname*{\inf}_{s\in [0, 1)}\rho_i(s)>0 &C_2:=\max_{i\in I_{1}\cup I_2}\operatorname*{\sup}_{s\in [0, 1)}\rho_i(s)>0
%&L:=\max_{s\in[0, 1)}e^{|w|(1+s)}>0, &K:=\sup_{s\in[0, 1)}\sum_{i\in I_1}|f_i(s)|^p\\
%&0<\epsilon_1<\frac{\epsilon}{2\|f_1\|}, &0<\epsilon_2<\frac{\epsilon}{2\|f_2\|}.
\end{align*}

Clearly, there is some $N_0\in\N$ such that for all $n\geq N_0$
\begin{equation}\label{eqq}\left(\bigcup_{i\in I_2}M_n(i)\right)\cap I_1=\emptyset, \qquad 
\left(\bigcup_{i\in I_1}K_n(i)\right)\cap I_2=\emptyset\end{equation}
for all $n\geq N_0.$

Given fixed $\delta_1, \delta_2>0$,  by $(iii)$  there exists  $n_h\in\N$, $n_{h}>N_0$,  such that 
for all $i\in I_1$ there is  $s_{i, n_h}\in [0,1)$  for which
    \begin{align*}
        &\frac{1}{\rho_{K_{n_h}(i)}(s_{i, n_h})}+        \left(\sum_{j\in M_{n_h}(K_{n_h}(i))}\frac{1}{\rho_j(s_{i, n_h})^{1/(p-1)}}\right)>\frac{1}{ \delta_1}\end{align*}
and,  for every $i\in I_2$, there exists $t_{i, n_h} \in [0,1)$ such that
    \begin{align*}&\left(\sum_{j\in M_{n_h}(i)}\frac{1}{\rho_j(t_{i, n_h})^{1/(p-1)}}\right)>\frac{1}{\delta_2}.
\end{align*}
%Consequently, for every $i\in I_1$
%\begin{align*}
%        &\frac{1}{\rho_{K_{n_h-1}(i)}(s_{i, n_h})}+
%        \left(\sum_{j\in M_{n_h}(K_{n_h}(i))}\frac{1}{\rho_j(s_{i, n_h})^{1/(p-1)}}\right)\\
%        \geq& \frac{1}{\rho_{K_{n_h-1}(i)}(s_{i, n_h})}+
%        \left(\sum_{j\in M_{n_h-1}(K_{n_h-1}(i))}\frac{1}{\rho_j(s_{i, n_h})^{1/(p-1)}}\right)>\frac{1}{ \delta_1}.\end{align*}
        Then, for every $i\in I_1$
$$\frac{1}{\rho_{K_{n_{h}}(i)}(s_{i, n_h})} >\frac{1}{2\delta_1} \qquad \text{or}\ 
        \left(\sum_{j\in M_{n_{h}}(K_{n_{h}}(i))}\frac{1}{\rho_j(s_{i, n_h})^{1/(p-1)}}\right)>\frac{1}{ 2\delta_1},$$
hence we can write $I_1=J_1\cup J_2$ where 
\begin{align*}
    &J_1:=\left\{i\in I_1 : \frac{1}{\rho_{K_{n_{h}}(i)}(s_{i, n_h})} >\frac{1}{2\delta_1}\right\},\\
    &J_2:=\left\{i\in I_1 : \left(\sum_{j\in M_{n_{h}}(K_{n_{h}}(i))}\frac{1}{\rho_j(s_{i, n_h})^{1/(p-1)}}\right)>\frac{1}{ 2\delta_1}\right\}.
\end{align*}
By Lemma \ref{lema sucesiones},  for each $i\in  J_2\setminus J_1$ there exists a family of numbers   $(u_{j}^i)_{j\in M_{n_h}(K_{n_{h}}(i))}$, $u^i_{j}\geq 0$
such that 
\begin{align*}
    \sum_{j\in M_{n_{h}}(K_{n_{h}}(i))}u_j^i=1
    \qquad \sum_{j\in M_{n_{h}}(K_{n_{h}}(i))}|u_j^i|^p\rho_j(s_{i,n_h})<(2\delta_1)^{p-1}
    ,
\end{align*}
and for every $i\in I_2$ there exists $(\nu^i_{j})_{j\in M_{n_h}(i)}$, $\nu_{j}^i\geq0$ such that    
\begin{align*}%\label{eqq1}
    \sum_{j\in M_{n_{h}}(i)}\nu^i_j=1,\qquad  
    &\sum_{j\in M_{n_{h}}(i)}|\nu^i_{j}|^p\rho_j(t_{i, n_h})<\delta_2^{p-1}.
\end{align*}
Now, define the function $g\in L^{p}_{\rho}(L(G))$ such that 
$$g_i(s):=\begin{cases}
    (f_1)_i(s)\hspace{6em} &\text{ if }\hspace{0.1em}i\in {J_1};\\
    (f_1)_i(s)(1-u_i^i) &\text{ if }i\in  J _2\setminus J_1;\\
    -(f_1)_l(s)u_i^l\hspace{3.1em} &\text{ if } i\in M_{n_h}(K_{n_h}(l)) \text{ for some }l\in  J_2\setminus J_1 \text{ and } i\notin I_1,;\\
    (f_2)_l(s)\nu^l_{i} \hspace{4em} &\text{ if  }i\in M_{n_{h}}(l)\text{ for some }l\in I_2;\\
    0\hspace{8.8em}&\text{ otherwise.}
  \end{cases}$$
If $l\in  J_1$ and $j\in M_{n_{h}}(K_{n_{h}}(l))$, then, by \eqref{eqq} and \eqref{int}, $j\notin M_{n_{h}}(I_2)$   and 
$j\notin  M_{n_{h}}( K_{n_h}(J_2\setminus J_1))$, 
so  
\begin{equation}\label{ec 3.3-14}
    (T_{n_{h}}g)_{K_{n_{h}}(l)}(s)=\sum_{j\in M_{n_{h}}(K_{n_{h}}(l))}g_j(s)=(f_1)_l(s).
\end{equation}
If $l\in J_2\setminus J_1$, then  for every $j\in M_{n_{h}}(K_{n_{h}}(l))$ it holds that $j\notin M_{n_{h}}(I_2)$ and 
\begin{equation}\label{ec 3.3-15}
    (T_{n_{h}}g)_{K_{n_{h}}(l)}(s)=\sum_{j\in M_{n_{h}}(K_{n_{h}}(l))}g_j(s)=(f_1)_l(s)-\sum_{j\in M_{n_{h}}(K_{n_{h}}(l))}(f_1)_l(s)u_j^l=0.
\end{equation}
If $l\in M_{n_{h}}(I_2)$ then 
\begin{equation}\label{ec 3.3-16}
    (T_{n_{h}}g)_{K_{n_{h}}(l)}(s)=\sum_{j\in M_{n_{h}}(K_{n_{h}}(l))}(f_2)_{K_{n_{h}}(l)}(s)\nu_j^{K_{n_h}(l)}=(f_2)_{K_{n_{h}}(l) }(s).
\end{equation}
Finally, if $i\in I\setminus\Big( K_{n_{h}}( J_1) \cup K_{n_{h}}(J_2 \setminus J_1) \cup K_{n_{h}}(M_{n_{h}}(I_2))\Big) =I\setminus K_{n_{h}}(I_1\cup M_{n_{h}}(I_2))$, then $M_{n_{h}}(i)$ has  empty intersection with the support of $g$, hence  $(T_{n_{h}}g)_i(s)=0$.\\ 
%Observe further that  the admissibility condition \eqref{1<p} yields that   for all $l\in  J_1$ and $s\in[0, 1)$ 
%$$\rho_{K_{n_{h}}(l)}(s)\leq Me^{w(s_{l, n_h}+1-s)}\rho_{K_{n_{h}-1}(l)}(s_{l, n_h})\leq Me^{2|w|}\rho_{K_{n_{h}-1}(l)}(s_{l, n_h})< 2Me^{2|w|}\delta_1.$$

By  \eqref{ec 3.3-14}, \eqref{ec 3.3-15} and \eqref{ec 3.3-16} we have: 
\begin{align*}
    &\|T_{n_{h}}g-f_2\|_{p,\rho}^p=\int_0^1\sum_{l\in  J_1} |(f_1)_l(s)|^p\rho_{K_{n_{h}}(l)}(s)ds\leq \int_0^1\sum_{l\in  J_1} |(f_1)_l(s)|^p\frac{C_2}{C_1}\rho_{K_{n_{h}}(l)}(s_{l, n_h})ds\\
    &\leq \frac{C_2}{C_1}2\delta_1 \int_0^1\sum_{l\in  J_1} |(f_1)_l(s)|^pds<\frac{C_2}{C_1^2}2\delta_1\int_0^1\sum_{l\in  J_1} |(f_1)_l(s)|^p \rho_l(s)ds = \frac{C_2}{C_1^2}2\delta_1||f_1||^p_{p,\rho} . 
\end{align*}
Moreover, 
\begin{align*}
    &\|f_1-g\|_{p,\rho}^p=\sum_{i\in I }\int_{0}^1|(f_1)_i(s)-g_i(s)|^p\rho_i(s)ds\\
    &=\sum_{l\in I_2}\sum_{j\in M_{n_{h}}(l)}\int_{0}^1|(f_2)_l(s)\nu_j^l|^p\rho_j(s)ds\\
    &\hspace{8em}+\sum_{l\in J_2\setminus J_1}\sum_{j\in M_{n_{h}}(K_{n_{h}}(l))}\int_0^1|(f_1)_l(s)u_j^l|^p\rho_j(s)ds\\
    &=\sum_{l\in I_2}\int_{0}^{1}|(f_2)_l(s)|^p\sum_{j\in M_{n_{h}}(l)}|\nu_{j}^l|^p\rho_j(s)ds+\\
    &\hspace{8em}\sum_{l\in  J_2\setminus J_1}\int_0^1|(f_1)_l(s)|^p\sum_{j\in M_{n_{h}}(K_{n_h}(l))}|u_j^l|^p\rho_j(s)ds\\
    &\leq \frac{C_2}{C_1}\sum_{l\in I_2}\int_{0}^1|(f_2)_l(s)|^p\left(\sum_{j\in M_{n_{h}}(l)}|\nu_{j}^l|^p\rho_j(t_{l, n_h})\right)ds+\\
    &\hspace{8em}\frac{C_2}{C_1}\sum_{l\in J_2\setminus J_1}\int_{0}^1|(f_1)_l(s)|^p\left(\sum_{j\in M_{n_{h}}(K_{n_{h}}(l))}|u_j^l|^p\rho_j(s_{l,n_h})\right)ds\\
    &\leq \frac{C_2}{C_1^2}\delta_2^{p-1}\sum_{l\in I_2}\int_{0}^1|(f_2)_l(s)|^p\rho_l(s)ds+
     \frac{C_2}{C_1^2}(2\delta_1)^{p-1}\sum_{l\in J_2\setminus J_1}\int_{0}^1|(f_1)_l(s)|^p\rho_l(s)ds\leq\\
    &\leq \frac{C_2}{C_1^2}\delta_2^{p-1} ||f_2||_{p,\rho}^p+
    (2\delta_1)^{p-1}\frac{C_2}{C_1^2}\|f_1\|^p_{p, \rho}.
\end{align*}
Now, choosing 
$$\delta_1=\left(\frac{\varepsilon^pC_1^2}{4C_2||f_1||_{p,\rho}^p}\right)^{1/(p-1)} \hspace{1em}\text{ and }\hspace{1em} \delta_2= \left(\frac{\varepsilon ^pC_1^2}{2C_2||f_2||^p_{p,\rho}}\right)^{1/(p-1)}$$
we get the assertion.

 Following a similar argument, it can be proved that if condition $(iii)$ holds, then given $\epsilon>0$ and $f_1, f_2, f_3, f_4\in F^p_\rho$ there exists some $k_0\in\mathds{N}$ and $g_1, g_2\in L^p_\rho(L(G))$ such that $$\|f_1-g_1\|_{p,\rho}<\epsilon,\hspace{2em} \|f_2-g_2\|_{p,\rho}<\epsilon$$
and also
$$\|T_{n_{k_0}}g_1-f_3\|_{p,\rho}<\epsilon,\hspace{2em}\|T_{n_{k_0}}g_2-f_4\|_{p,\rho}<\epsilon,$$
hence $\{T_t\}_{t\geq 0}$ is weakly mixing. 
 \end{proof}

 \begin{remark}
    As in the rooted case, one can show that either on $L_{\rho}^1(L(G))$ or on $L_{\rho}^p(L(G))$, the translation semigroup is mixing if and only if condition $(iii)$ of Theorem \ref{TH3} holds for the full sequence $(n)_{n\in\mathds{N}}$.
\end{remark}

 \begin{remark} In the case considered in Remark \ref{Ltr}, the conditions in Theorems \ref{weakly mixing translation} and \ref{TH3}  are exactly those ensuring that the left translation semigroup is hypercyclic when acting on $[0,+\infty[$ or on $\R$ (see e.g. \cite[Example 7.10]{Alfred} and \cite{desch_schappacher_webb1997hypercyclic}).
\end{remark}

\end{document}